\documentclass[11pt]{amsart}

\usepackage{color,enumitem}
\usepackage{amsmath,amssymb,bbm,bm}
\usepackage{amscd,tikz-cd}
\usepackage{mathrsfs}
\usepackage{amsthm}
\usepackage{mathtools, todonotes}

\theoremstyle{plain}

\newtheorem{theorem}{\bf Theorem}[section]
\newtheorem{lemma}[theorem]{\bf Lemma}
\newtheorem{proposition}[theorem]{\bf Proposition}

\theoremstyle{definition}
\newtheorem{definition}[theorem]{\bf Definition}
\newtheorem{example}[theorem]{\bf Example}
\newtheorem{remark}[theorem]{\bf Remark}

\newtheorem{question}[theorem]{\bf Question}

\newcommand{\disp}{\displaystyle}

\newcommand{\eqa}[1]{
\begin{align*}
#1
\end{align*}}

\newcommand{\ri}{{\rm i}}

\newcommand{\E}{\mathrm{E}}

\newcommand{\N}{\mathbb{N}}
\newcommand{\Z}{\mathbb{Z}}
\newcommand{\Ad}{\mathrm{Ad}}
\newcommand{\Aut}{\mathrm{Aut}}

\usepackage[all]{xy}

\title[Kirchberg's lemma and approximately inner automorphisms]{An application of Kirchberg's lemma on central sequence algebras to groups of approximately inner automorphisms}
\author[H. Ando]{Hiroshi Ando}
\address{Hiroshi Ando, Department of Mathematics and Informatics, Chiba University, 1-33 Yayoi-cho, Inage, Chiba, 263- 8522,
Japan}
\email{hiroando@math.s.chiba-u.ac.jp}

\author[M. Doucha]{Michal Doucha}
\address{Institute of Mathematics\\
Czech Academy of Sciences\\
\v Zitn\'a 25\\
115 67 Praha 1\\
Czech Republic}
\email{doucha@math.cas.cz}

\begin{document}

\maketitle
\bibliographystyle{siam}
\begin{center}
    {\it Dedicated to the memory of Professor Eberhard Kirchberg}
\end{center}
\begin{abstract}
We revisit a well-known ``surjectivity onto quotient" type lemma of Kirchberg on the central sequence algebra of a separable unital {\rm C}$^*$-algebra, and use it to prove a ``surjectivity onto quotient" result on approximately inner automorphisms of a separable unital {\rm C}$^*$-algebra of stable rank one, which we can partially upgrade also to the non-separable case.
\end{abstract}
\section{Introduction}
Let us recall the following well-known 
lemma, proved by Kirchberg in an unpublished preprint, on the central sequence algebra of a separable unital C$^*$-algebra. 
\begin{proposition}\cite[Lemma 2.6]{Kirchberg2003}\label{prop kirchberg lifting of cs}
        Let $A$ be a unital {\rm C}$^*$-algebra, $J$ be a nonzero proper closed two-sided ideal of $A$ and $B$ be a separable {\rm C}$^*$-subalgebra of $A_{\omega}$. Then the following statements hold. 
        \begin{list}{}{}
            \item[{\rm (i)}] The ultrapower map $(\pi_J)_{\omega}$ of the quotient map $\pi_J\colon A\to A/J$ induces the following exact sequence:
        \[0\longrightarrow J_{\omega}\xrightarrow{\iota_{\omega}} A_{\omega}\xrightarrow{(\pi_J)_{\omega}}(A/J)_{\omega}\longrightarrow 0.\]
        Here, $\iota\colon J\to A$ is the inclusion map. 
        \item[{\rm (ii)}] There exists a positive contraction $e\in B'\cap J_{\omega}$ such that $eb=b=be$ for all $b\in B\cap J_{\omega}$.  
        \item[{\rm (iii)}] The sequence 
        \begin{equation}
            0\longrightarrow B'\cap J_{\omega}\xrightarrow{\iota_{\omega}} B'\cap A_{\omega}\xrightarrow{(\pi_J)_{\omega}}(\pi_J)_{\omega}(B)'\cap (A/J)_{\omega}\longrightarrow 0
        \end{equation}
        is also exact.  
        \end{list}
\end{proposition} 
Since it will play a key role in this notes but is unfortunately never published in a refereed journal, we include his proof in the appendix. 
Setting $B=A$, (iii) shows in particular that the natural map $A'\cap A_{\omega}\to (A/J)'\cap (A/J)_{\omega}$ induced by the quotient map $A\to A/J$ is onto. The proof is elementary in the sense that it just uses strictly positive elements and quasi-central approximate units. It is however of fundamental importance to the study of central sequence algebras, especially for non-simple C$^*$-algebras. It is also worth noticing that many important structural results regarding the central sequence algebra can be shown in the very similar fashion using basic tools, as Kirchberg  demonstrated e.g. in the work on his central sequence algebra $F_{\omega}(A)$ for non-unital C$^*$-algebra \cite{kirchbergAbelMR2265050}.   
Even though \cite{Kirchberg2003} has never been published in a refereed journal, this lemma (and other results from it) has been used for the structural study of central sequence algebras. Less obviously, it has some application to the central sequence algebra of $W^*$-algebras as well. For example, the first named author and Kirchberg showed \cite{MR3532173}, generalizing the above surjectivity result to the $F_{\omega}(A)$ setting, that whenever $A$ is a separable simple non-type I C$^*$-algebra, then $F_{\omega}(A)$ contains uncountably many nonseparable type III factors as its sub-quotients, indicating that for a large class of C$^*$-algebras, $F_{\omega}(A)$ is quite large. This asserts in particular that even though the central sequence algebra of the group von Neumann algebra $L(\mathbb{F}_2)$ of the free group $\mathbb{F}_2$ is trivial, the central sequence C$^*$-algbera of the reduced group C$^*$-algebra $C_r^*(\mathbb{F}_2)$ is highly non-commutative and huge, which answered a question of Kirchberg himself in \cite{kirchbergAbelMR2265050}. 
In this notes, we apply Proposition~\ref{prop kirchberg lifting of cs} to prove a surjectivity-onto-quotient type result in the context of the groups of approximately inner automorphisms. To explain in more detail, we introduce some terminology following \cite{Rob19}. 
Let $A$ be a unital ${\rm C}^*$-algebra. Denote by $U(A)$ the group of unitary elements of $A$ equipped with the norm topology and by $U_A$ its connected component of the identity. By $V_A$ we denote the group of approximately inner automorphisms induced by $U_A$. That is, \[V_A:=\overline{\{\Ad(u)\mid u\in U_A\}}\subseteq \Aut(A),\] where for every $u\in U(A)$, $\Ad(u)\in\Aut(A)$ is the corresponding inner automorphism, and the group $\Aut(A)$ is equipped with the pointwise convergence topology. 
Similarly, if $I\subseteq A$ is a closed two-sided ideal, by $U(\tilde{I})$, resp. $U_I$, resp. $V_I$ we denote the unitary group of the unitization $\tilde{I}$ of $I$, resp. normal subgroups $U_{\widetilde{I}}$, resp. $\overline{\{\Ad(u)\colon u\in U_I\}}\unlhd V_A$.

The group $V_A$ plays an important role in determining the structure of $\Aut(A)$ as we have the normal series $V_A\unlhd V(A):=\overline{\{\Ad(u)\mid u\in U(A)\}}\unlhd \Aut(A)$. The quotients $V(A)/V_A$ and $\Aut(A)/V(A)$ are in some cases tractable, see e.g. \cite{ElRo}, where the case when $A$ is a real rank zero unital inductive limit of circle algebras is treated and $\Aut(A)/V(A)$ and $V(A)/V_A$ are described. Therefore, it is of interest to understand the structure of $V_A$ itself better. Robert in \cite{Rob19} proves that $V_A$ is (topologically) simple if $A$ is simple and a more detailed description of the normal subgroup structure of $V_A$, for a general unital $A$, is obtained in \cite{AnDo}, where the following main result of this paper finds an important application. We fix a free ultrafilter $\omega$ on $\N$.

\begin{theorem}\label{thm:main}
Let $A$ be a separable unital ${\rm C}^*$-algebra and $I\subseteq A$ be a closed proper two-sided ideal with connected unitary group $U(\tilde{I})$.
Assume moreover that either the central sequence algebra $A'\cap A_{\omega}$ has stable rank one or the unitary group $U((A/I)'\cap (A/I)_{\omega})$ is connected. 
Then the canonical quotient map $P:V_A\to V_{A/I}$ satisfies
\begin{enumerate}
    \item\label{it:1} $\mathrm{Ker}(P)=V_I$;
    \item\label{it:2} $P$ naturally factorizes as $\tilde P\circ Q$, where $Q:V_A\to V_A/V_I$ is the quotient map and $\tilde P:V_A/V_I\to V_{A/I}$ is a topological group isomorphism of $V_A/V_I$ onto its image. Consequently, $P[V_A]=V_{A/I}$ holds.
\end{enumerate}
\end{theorem}

If $A$ is a separable unital AF algebra, then $A$ and $I$ satisfy all the hypotheses in Theorem \ref{thm:main}. Indeed, if $A$ is a separable unital AF algebra, then $A'\cap A_{\omega}$ has stable rank one (see e.g. the paragraph after \cite[Question 5.2]{MR3908669BBSTWW}) and moreover $U(A'\cap A_{\omega})$ is connected (this is likely a folklore as well, but we include the proof in Proposition \ref{prop F(AF) has connected unitary group} as we could not find a proper reference).
As a consequence, we get the following result for separable unital AF-algebras.
\begin{theorem}\label{thm:AFconsequence}
Let $A$ be a separable unital AF-algebra and $I\subseteq A$ be any proper closed two-sided ideal. Then the canonical quotient map $P:V_A\to V_{A/I}$ satisfies that $\mathrm{Ker}(P)=V_I$ and $P$ naturally factorizes as $\tilde P\circ Q$, where $Q:V_A\to V_A/V_I$ is the quotient map and $\tilde P:V_A/V_I\to V_{A/I}$ is a topological isomorphism.
\end{theorem}
Theorem~\ref{thm:AFconsequence} is partially generalized to general unital locally AF algebras in $\S$\ref{subsec nonseparable} which is then applied in \cite{AnDo} to obtain a precise description of closed normal subgroups of $V_A$, where $A$ is any unital locally AF algebra.

Note that it is well-known that the identity component of the unitary group of a unital C$^*$-algebra is algebraically generated by exponentials of self-adjoint elements. Thus, it is clear that the map $P\colon V_A\to V_{A/I}$ has dense image and ${\rm Ker}(P)$ contains $V_I$. However, it is not obvious that the kernel of $P$ is precisely $V_I$ and that it is onto if $A$ is separable. We are grateful to the anonymous referee and to Leonel Robert who found gaps in the proof of Theorem~\ref{thm:main} in the previous version of the paper which forced us to add additional assumptions to the statement of the theorem, namely of ${\rm sr}(A'\cap A_{\omega})=1$ and $U(\widetilde{I})$ being connected.
 
It would be interesting to decide the equality of the kernel of $P$ with $V_I$ in general unital separable C$^*$-algebra $A$ as it is equivalent to a statement concerning the closed normal subgroup structure of $V_A$ (we refer to \cite{AnDo}). 

\begin{remark}\label{rem P onto undecidable}
After the paper has been refereed, we were informed from Ilijas Farah that the surjectivity question of $P$ is undecidable within axioms of ZFC for some non-separable $A$.  Indeed, take $A=\mathbb{B}(H)$ for separable infinite-dimensional $H$ and $I=\mathbb{K}(H)$. Under the 
continuum hypothesis (CH), Phillips--Weaver \cite{PhillipsWeaverMR2322680} showed that the Calkin algebra $Q(H)=\mathbb{B}(H)/\mathbb{K}(H)$ has (approximately inner) outer automorphisms. Moreover, from their arguments it can be shown that such automorphisms can be chosen to be the limit of inner automorphisms by unitaries with trivial Fredholm index (see also \cite[$\S$17]{FarahbookMR3971570} for alternative argument). 
Since $V_{\mathbb{B}(H)}$ is exactly the inner automorphisms, $P$ is not onto under CH. On the other hand, under open coloring axiom (OCA), Farah \cite{FarahCalkinMR2776359} showed that all automorphisms of $Q(H)$ are inner, whence $P$ is onto. 
\end{remark}

The paper is organized as follows. In $\S$\ref{subsec separablecase} we prove Theorem~\ref{thm:main}, where Kirchberg's lemma is directly applicable and sequences suffice for topological arguments. We show in this case that $\tilde{P}$ is a topological group isomorphism, for appropriate $A$ and $I$ covered by Theorem~\ref{thm:main}, of $V_A/V_I$ onto image. Thus, $P$ has dense image which is itself a Polish group. It follows that $P$ is onto. In $\S$\ref{subsec nonseparable} we prove a partial generalization of Theorem~\ref{thm:AFconsequence} in the non-separable case.

Finally, in $\S$\ref{sec Kirchberg central sequence} we also remark that a property of unital Kirchberg algebras discovered by Kirchberg himself can be used to show that the tensor product of $*$-homomorphisms $\theta_i\colon M_i\to N_i\,(i=1,2)$ between von Neumann algebras does not always extend to a $*$-homomorphism to $M_1\overline{\otimes}M_2\to N_1\overline{\otimes}N_2$, even though it does admit an extension as a unital completely positive map, as was shown by Nagisa--Tomiyama \cite{NagisaTomiyamaMR620290}.     

\section{Main Results}
The goal of this section is to prove Theorem~\ref{thm:main}. We first prove the theorem in the separable case and then upgrade to possibly non-separable locally AF algebras. We start with few preliminary results.
\begin{lemma}\label{lem:approxInnquotient}
Let $A$ be a unital {\rm C}$^*$-algebra. For every $\phi\in V_A$ and every two-sided ideal $I$ we have $\phi[V]\subseteq I$, thus $\phi$ naturally induces $\pi_I\circ \phi=:P(\phi)\in V_{A/I}$, where $\pi_I:A\rightarrow A/I$ is the quotient map. The map $P: V_A\rightarrow V_{A/I}$ is a continuous homomorphism satisfying $P(\mathrm{Ad}(u))=\mathrm{Ad}(\pi_I(u))$.
\end{lemma}
\begin{proof}
Pick $\phi\in V_A$, an ideal $I$ and $x\in I$. Since $\phi$ is a pointwise limit of $(\mathrm{Ad}(u_i))_i$, where $(u_i)_i\subseteq U_A$, we have \[\phi(x)=\lim_i u_ixu_i^*\in I\] since $I$ is two-sided and closed. It follows that \[P(\phi)(x+I):=\phi(x)+I=\pi_I\circ \phi(x)\] is a well-defined automorphism of $A/I$. Let $u\in U_A$ and $x\in A$. By definition, we clearly have \[P(\mathrm{Ad}(u))(x+I)=\pi_I(uxu^*)=\pi_I(u)(x+I)\pi_I(u)^*=\mathrm{Ad}(\pi_I(u))(x+I).\] $P$ is obviously continuous, thus for any $\phi\in V_A$ we have $P(\phi)\in V_{A/I}$ since if $\phi$ is a limit of $(\mathrm{Ad}(u_i))_i$, then $P(\phi)$ is a limit of $(\mathrm{Ad}(\pi_I(u_i)))_i$ in $V_{A/I}$. It is clear that $P$ is a homomorphism.
\end{proof}
Let $A$ be a {\rm C}$^*$-algebra, and $\omega$ be a free ultrafilter on $\N$. 
Let $\ell^{\infty}(A)$ be the C$^*$-algebra of all bounded sequences of elements in $A$ with the norm $\|a\|=\sup_n\|a_n\|,\,a=(a_1,a_2,\dots)\in \ell^{\infty}(A)$. 
The set 
$$c_{\omega}(A)=\left \{a\in \ell^{\infty}(A)\,\middle|\, \lim_{n\to \omega}\|a_n\|=0\right \}$$
 is a closed two-sided ideal of $\ell^{\infty}(A)$. The quotient C$^*$-algebra $A_{\omega}=\ell^{\infty}(A)/c_{\omega}(A)$ is called the ultrapower of $A$ along $\omega$. $A$ is naturally identified with the C$^*$-subalgebra of $A_{\omega}$ by the diagonal embedding $A\ni a\mapsto (a,a,\dots)+c_{\omega}(A)\in A_{\omega}$. The relative commutant $A'\cap A_{\omega}$ is called the central sequence algebra of $A$. 
 When there is no danger of confusion, we use the abbreviation $(v_n)_{\omega}$ for an element in $A_{\omega}$ represented by $(v_n)_{n=1}^{\infty}\in \ell^{\infty}(A)$:
 \[(v_n)_{\omega}=(v_1,v_2,\dots)+c_{\omega}(A)\in A_{\omega}\]

\subsection{The proof of Theorem~\ref{thm:main}}\label{subsec separablecase}

We will use the next lemma, which is immediate from Kirchberg's surjectivity lemma, Proposition \ref{prop kirchberg lifting of cs}. 
\begin{lemma}\label{lem lifting of U_(F(A/I))}
Let $A$ be a separable unital {\rm C}$^*$-algebra, $I$ be a closed two-sided proper ideal of $A$. Then for each $v\in U_{(A/I)'\cap (A/I)_{\omega}}$, there exists $u\in U_{A'\cap A_{\omega}}$ such that $(\pi_I)_{\omega}(u)=v$ holds. 
\end{lemma}
\begin{proof}
Since $v\in U_{(A/I)'\cap (A/I)_{\omega}}$, there exist self-adjoint elements $b_1,\dots,b_n\in (A/I)'\cap (A/I)_{\omega}$ such that $v=e^{\ri b_1}\cdots e^{\ri b_n}$. By Proposition \ref{prop kirchberg lifting of cs}, there exists $a_1',\dots, a_n'\in A'\cap A_{\omega}$ such that $(\pi_I)_{\omega}(a_k')=b_k,\,k=1,\dots,n$. Then $a_k=\frac{1}{2}\{a_k'+(a_k')^*\},\,k=1,\dots,n$ are self-adjoint elements in $A'\cap A_{\omega}$, so that $u=e^{\ri a_1}\cdots e^{\ri a_n}\in U_{A'\cap A_{\omega}}$ satisfies $(\pi_I)_{\omega}(u)=v$. 
\end{proof}

\begin{lemma}\label{lem un close to vn in U_I}  Let $A$ be a unital {\rm C}$^*$-algebra, $I$ be a closed two-sided proper ideal of $A$. 
    If $(u_n)_{n=1}^{\infty}$ is a sequence in $U(A)$ such that ${\rm dist}(u_n,\tilde{I})\xrightarrow{n\to \infty}0$, 
    then there exists a sequence $(v_n)_{n=1}^{\infty}$ in $U(I)$ such that $\disp \lim_{n\to \infty}\|u_n-v_n\|=0$. 
\end{lemma}
\begin{proof}
By assumption, there exists a sequence $(x_n)_{n=1}^{\infty}$ in $\widetilde{I}$ such that $\disp \lim_{n\to \infty}\|u_n-x_n\|=0$. Then also $\disp \lim_{n\to \infty}\|u_n^*-x_n^*\|=0$ and thus $\disp \lim_{n\to \infty}\|u_n^*u_n-x_n^*x_n\|=\lim_{n\to \infty}\|1-x_n^*x_n\|=0$. Then by functional calculus, we have also $\disp \lim_{n\to \infty}\|1-|x_n|\|=0$. In particular, $x_n\in {\rm GL}(\widetilde{I})$ eventually, and thus we may assume that all $x_n$ are invertible. Then it also holds that $\||x_n|^{-1}-1\|\xrightarrow{n\to \infty}0$. Then $v_n=x_n|x_n|^{-1}\in U(I)$, and 
    \eqa{
        \|u_n-v_n\|&\le \|u_n-x_n\|+\|x_n-x_n|x_n|^{-1}\|\\
        &\le \|u_n-x_n\|+\|x_n\|\|1-|x_n|^{-1}\|\\
        &\xrightarrow{n\to \infty}0.
    }
\end{proof}
\begin{proof}[Proof of Theorem~\ref{thm:main}]
    First, we show that ${\rm Ker}(P)=V_I$. It is clear that $V_I\subset {\rm Ker}(P)$. Let $\phi\in {\rm Ker}(P)$. Then because $\phi\in V_A$ and $A$ is separable, there exists a sequence $(u_n)_{n=1}^{\infty}$ in $U_A$ such that $\disp \phi=\lim_{n\to \infty}{\rm Ad}(u_n)$ in ${\rm Aut}(A)$. Let $v_n=\pi_I(u_n)\in U_{A/I},\,n\in \N$. 
    Then by $P(\phi)={\rm id}$, for all $b\in A/I$, we have $\disp \lim_{n\to \infty}\|v_nbv_n^*-b\|=\lim_{n\to \infty}\|v_nb-bv_n\|=0$. 
    
    Let 
    $$v=(v_1,v_2,\dots)+c_{\omega}(A/I)\in (A/I)_{\omega}.$$
    Then $v\in U\big((A/I)'\cap (A/I)_{\omega}\big)$.\medskip
    
    \noindent{\bf Claim.} There exists $\tilde u\in U(A'\cap A_\omega)$ such that $(\pi_I)_\omega(\tilde u)=u$.\medskip

    If $U((A/I)'\cap (A/I)_{\omega})$ is connected, then  by Lemma~\ref{lem lifting of U_(F(A/I))},
    there exists a unitary $\tilde u\in U_{A'\cap A_{\omega}}$ such that $(\pi_I)_\omega(\tilde u)=v$. So assume now that $A'\cap A_{\omega}$ has stable rank one. First, applying again Proposition~\ref{prop kirchberg lifting of cs} (iii), there exists a contraction (not necessarily a unitary) $t\in A'\cap A_{\omega}$ such that $(\pi_I)_{\omega}(t)=v$, and second, by Proposition~\ref{prop unitarypd}, there exists a unitary $\tilde u\in A'\cap A_{\omega}$ such that $t=\tilde u|t|$. Let $(t_n)_{n=1}^{\infty}$ be a sequence of contractions in $A$ representing $t$ and let $(\tilde u_n)_{n=1}^\infty$ be a sequence of unitaries in $A$ representing $\tilde u$. Also let $u=(u_n)_\omega\in A_\omega$.
    
    We have $\pi_{\omega}((t^*t)^{\frac{1}{2}})=(v^*v)^{\frac{1}{2}}=1$, so $|t|-1\in I_{\omega}$. Therefore $\lim_{n\to \omega}\|\pi_I(1-|t_n|)\|=0$ and so we obtain \eqa{
        \|(\pi_I)_{\omega}(\tilde{u})-(\pi_I)_{\omega}(t)\|&=\lim_{n\to \omega}\|\pi_I(\tilde{u}_n)\pi_I(1-|t_n|)\|\\
        &=\lim_{n\to \omega}\|\pi_I(1-|t_n|)\|\\
        &=0.
    }
    
    This shows that $(\pi_I)_{\omega}(\tilde{u})=(\pi_I)_{\omega}(t)=v=(\pi_I)_{\omega}(u)$. This finishes the proof of the claim.$\qed$

    Thus, $u\tilde{u}^*-1\in {\rm Ker}((\pi_I)_{\omega})=I_{\omega}$ by Proposition \ref{prop kirchberg lifting of cs}(i), i.e., we have 
    $${\rm dist}(u\tilde{u}^*,\tilde{I}_{\omega})=\lim_{n\to \omega}{\rm dist}(u_n\tilde{u}_n^*,\tilde{I})=0.$$
    Since $A$ is separable and $\tilde{u}\in A'\cap A_{\omega}$, we may choose a subsequence $n_1<n_2<\cdots $ such that 
    \begin{align}
        \lim_{k\to \infty}\|\tilde{u}_{n_k}a-a\tilde{u}_{n_k}\|=0,\,\,\,a\in A,\\
        \lim_{k\to \infty}{\rm dist}(u_{n_k}\tilde{u}_{n_k}^*,\tilde{I})=0. 
    \end{align}
    By Lemma \ref{lem un close to vn in U_I}, we may find a sequence $(w_k)_{k=1}^{\infty}$ in $U_I$ such that 
    \[\lim_{k\to \infty}\|u_{n_k}\tilde{u}_{n_k}^*-w_k\|=0.\]
    We now claim that 
    \[\lim_{k\to \infty}w_kaw_k^*=\lim_{k\to \infty}u_{n_k}au_{n_k}^*(=\phi(a)),\,\,\,a\in A,\]
    that is, $\disp \phi=\lim_{k\to \infty}{\rm Ad}(w_k)\in V_I$ holds, which would finish the proof. 
    To show the claim, let $a\in A$. Then 
    \eqa{
        w_kaw_k^*&=(w_k-u_{n_k}\tilde{u}_{n_k}^*)aw_k^*+u_{n_k}\tilde{u}_{n_k}^*a(w_k^*-\tilde{u}_{n_k}u_{n_k}^*)\\
        &\hspace{1.0cm}+u_{n_k}(\tilde{u}_{n_k}^*a\tilde{u}_{n_k}-a)u_{n_k}^*+u_{n_k}au_{n_k}^*,    
    }
    whence 
    \eqa{
        \|w_kaw_k^*-u_{n_k}au_{n_k}^*\|&\le \|w_k-u_{n_k}\tilde{u}_{n_k}^*\|\|a\|+\|a\|\,\|w_k^*-\tilde{u}_{n_k}u_{n_k}^*\|\\
        &\hspace{1.0cm}+\|a\tilde{u}_{n_k}-\tilde{u}_{n_k}a\|\\
        &\xrightarrow{n\to \infty}0.
    }
    Next, we show that $P\colon V_A\to V_{A/I}$ is a surjection. Let $G=V_A/V_I$ be the quotient group with the quotient topology. Let $q_I\colon V_A\to G$ be the canonical quotient map. 
    Since ${\rm Ker}(P)=V_I$, $P$ induces a continuous injective homomorphism $\tilde{P}\colon G\to V_{A/I}$. Since $A$ is separable, $V_A$ and $G$ are Polish groups. 
    We show that $\tilde{P}$ is an open mapping, whence a topological isomorphism onto its image. This would imply that the image $P(V_A)$ is a dense Polish subgroup of the Polish group $V_{A/I}$, thus $P(V_A)=\tilde{P}(G)=V_{A/I}$. 
        
    To show that $\tilde{P}$ is open, we show that whenever a sequence $(\phi_k)_{k=1}^{\infty}$ in $V_A$ satisfies $P(\phi_k)\to {\rm id}$ in $V_{A/I}$, then there exists a subsequence $k_1<k_2<\cdots$ and $(\chi_{i})_{i=1}^{\infty}$ in $V_I$ such that 
    $\chi_{i}\circ \phi_{k_i}\to {\rm id}$ in $V_A$. Note that the last condition is equivalent to $[\phi_{k_i}]\to {\rm id}$ in $G$, where $[\phi]$ is the image of $\phi$ in $G$. 
    It is well-known that in a metric space $X$, if a sequence $(x_n)_{n=1}^{\infty}$ has the property that there exists $x\in X$ such that for any subsequence $(x_{n_k})_{k=1}^{\infty}$ there is a further subsequence $(x_{n_{k_i}})_{i=1}^{\infty}$ converging to $x$, then $(x_n)_{n=1}^{\infty}$ itself converges to $x$. Thus the above result would imply that whenever a sequence $(\phi_n)_{n=1}^{\infty}$ in $V_A$ satisfies $P(\phi_n)\to {\rm id}$ in $V_{A/I}$, then $[\phi_n]\to {\rm id}$ in $G$, which is exactly the openness of $\tilde{P}$. The proof is very similar to (but slightly more involved than) the proof of ${\rm Ker}(P)=V_I$. 
    
    Fix two free ultrafilters $\omega,\omega'\in \beta \N\setminus \N$. 
    Recall that on a partially ordered set $\N^2$ with the partial ordering given by $(n,m)\le (n',m')$ if and only if $n\le n'$ and $m\le m'$, the cofinal ultrafilter $\omega'\otimes \omega$ is defined by 
    \[\omega'\otimes \omega=\{A\subset \N^2\mid \{k\mid \{n\mid (k,n)\in A\}\in \omega\}\in\omega'\}\] 
    Then we have for any doubly indexed sequence in a compact Hausdorff space the following ultralimit formula
    \[\lim_{(k,n)\to \omega'\otimes \omega}x_{k,n}=\lim_{k\to \omega'}\lim_{n\to \omega}x_{k,n}.\]
    Consequently, for any C$^*$-algebra $B$, we have a natural identification 
    \[B_{\omega'\otimes \omega}=(B_{\omega})_{\omega'}.\]
    Thus the $\omega'$-ultrapower of the $\omega$-ultrapower of a C$^*$-algebra $B$ is nothing but the ultrapower of $B$ in another cofinal ultrafilter. Thus, Proposition \ref{prop kirchberg lifting of cs} can be applied to iterated ultrapower of separable C$^*$-algebras.

    For each $k\in \N$, by $\phi_k\in V_A$, there exists a sequence $(u_n^{(k)})_{n=1}^{\infty}$ in $U_A$ such that $\disp \phi_k=\lim_{n\to \infty}{\rm Ad}(u_n^{(k)})$ in $V_A$. Let $u^{(k)}=(u_n^{(k)})_{\omega}\in A_{\omega}$, $v_n^{(k)}=\pi_I(u_n^{(k)})\in A/I$ and $v^{(k)}=(v_n^{(k)})_{\omega}\in (A/I)_{\omega}$. Then for each $b\in A/I$, we have 
    \eqa{
        \|v^{(k)}b(v^{(k)})^*-b\|_{(A/I)_{\omega}}&=\lim_{n\to \omega}\|v_n^{(k)}b(v_n^{(k)})^*-b\|_{A/I}\\
        &=\|P(\phi_k)(b)-b\|_{A/I}\\
        &\xrightarrow{k\to \infty}0.
    }
    In particular, we have a unitary
    $$v=(v^{(k)})_{\omega'}\in (A/I)'\cap ((A/I)_{\omega})_{\omega'}=(A/I)'\cap (A/I)_{\omega'\otimes \omega}.$$
    By the same argument as in the proof of ${\rm Ker}(P)\subset V_I$, using either the stable rank one assumption on $A'\cap A_\omega$ or the connectedness of $U\big((A/I)'\cap (A/I)_\omega\big)$, and Proposition~\ref{prop kirchberg lifting of cs}(iii), we may find a doubly indexed sequence $(\tilde{u}_n^{(k)})_{n,k=1}^{\infty}$ in $U(A)$ such that 
    $\tilde{u}=(\tilde{u}_n^{(k)})_{\omega'\otimes \omega}\in A'\cap A_{\omega'\otimes \omega}$, such that $(\pi_I)_{\omega'\otimes \omega}(\tilde{u})=v.$ Let $u=(u_n^{(k)})_{\omega'\otimes \omega}\in A_{\omega'\otimes \omega}$. Then by Proposition \ref{prop kirchberg lifting of cs}(i), we have 
    \[u\tilde{u}^*-1\in {\rm Ker}((\pi_I)_{\omega'\otimes \omega})=I_{\omega'\otimes \omega},\]
    whence 
    \[{\rm dist}(u\tilde{u}^*,\tilde{I}_{\omega'\otimes \omega})=\lim_{(k,n)\to \omega'\otimes \omega}{\rm dist}(u_n^{(k)}(\tilde{u}_n^{(k)})^*,\tilde{I})=0.\]
    Then by Lemma \ref{lem un close to vn in U_I}, there exists $(w_n^{(k)})_{n,k=1}^{\infty}$ in $U(
I)=U_I$ such that 
    \[\lim_{k\to \omega'}\lim_{n\to \omega}\|u_n^{(k)}(\tilde{u}_n^{(k)})^*-w_n^{(k)}\|=0.\]
    Let $\{a_1,a_2,\dots\}$ be a countable dense subset of the closed unit ball of $A$. Then for each $i\in \N$ there exists $k_i$ such that $k_1<k_2<\cdots$, and the following two conditions hold.
    \eqa{
        \lim_{n\to \omega}\|\tilde{u}_n^{(k_i)}a_{\ell}-a_{\ell}\tilde{u}_{n}^{(k_i)}\|&<\frac{1}{i},\,(\ell=1,\dots,i)\\
        \lim_{n\to \omega}\|u^{(k_i)}_n(\tilde{u}^{(k_i)}_n)^*-w_n^{(k_i)}\|&<\frac{1}{i}.
    }
    Then choose $n_1<n_2<\cdots$ such that 
    \eqa{
        \|\tilde{u}_{n_i}^{(k_i)}a_{\ell}-a_{\ell}\tilde{u}_{n_i}^{(k_i)}\|&<\frac{1}{i},\,(\ell=1,\dots,i),\\
        \|u^{(k_i)}_{n_i}(\tilde{u}^{(k_i)}_{n_i})^*-w_{n_i}^{(k_i)}\|&<\frac{1}{i},\\
        \|\phi_{k_i}(a_{\ell})-u_{n_i}^{(k_i)}a_{\ell}(u_{n_i}^{(k_i)})^*\|&<\frac{1}{i}\,\,(\ell=1,\dots,i).
    }
    Let $\chi_i={\rm Ad}((w_{n_i}^{(k_i)})^*)\in V_I,\,i\in \N$.\\ \\
    \textbf{Claim.} The following equality holds. 
    \[\lim_{i\to \infty}\chi_i\circ \phi_{k_i}={\rm id}\,\,{\rm in}\,V_A.\]
    By the density of $\{a_1,a_2,\dots\}$ in the closed unit ball of $A$, it suffices to show that $\disp \lim_{i\to \infty}\chi_i\circ \phi_{k_i}(a_{\ell})=a_{\ell}$ for all $\ell\in \N$. Fix $\ell\in \N$. Then for each $i\in \N$, 
    \eqa{
        \|\chi_{i}\circ \phi_{k_i}(a_{\ell})-a_{\ell}\|&\le \|(w_{n_i}^{(k_i)})^*\big(\phi_{k_i}(a_{\ell})-u_{n_i}^{(k_i)}a_{\ell}(u_{n_i}^{(k_i)})^*\big)w_{n_i}^{(k_i)}\|\\
        &\hspace{0.5cm}+\|(w_{n_i}^{(k_i)})^*u_{n_i}^{(k_i)}a_{\ell}(u_{n_i}^{(k_i)})^*w_{n_i}^{(k_i)}-a_{\ell}\|\\
        &=\|\phi_{k_i}(a_{\ell})-u_{n_i}^{(k_i)}a_{\ell}(u_{n_i}^{(k_i)})^*\|+\|u_{n_i}^{(k_i)}a_{\ell}(u_{n_i}^{(k_i)})^*-w_{n_i}^{(k_i)}a_{\ell}(w_{n_i}^{(k_i)})^*\|\\
        &<\tfrac{1}{i}+\|u_{n_i}^{(k_i)}(\tilde{u}_{n_i}^{(k_i)})^*\big(\tilde{u}_{n_i}^{(k_i)}a_{\ell}-a_{\ell}\tilde{u}_{n_i}^{(k_i)}\big)(u_{n_i}^{(k_i)})^*\|\\
        &\hspace{0.5cm}+\|(u_{n_i}^{(k_i)}(\tilde{u}_{n_i}^{(k_i)})^*-w_{n_i}^{(k_i)})a_{\ell}\tilde{u}_{n_i}^{(k_i)}(u_{n_i}^{(k_i)})^*\|\\
        &\hspace{1.0cm}+\|w_{n_i}^{(k_i)}a_{\ell}(\tilde{u}_{n_i}^{(k_i)}(u_{n_i}^{(k_i)})^*-(w_{n_i}^{(k_i)})^*)\|\\
        &<\tfrac{1}{i}+\|\tilde{u}_{n_i}^{(k_i)}a_{\ell}-a_{\ell}\tilde{u}_{n_i}^{(k_i)}\|+\|u_{n_i}^{(k_i)}(\tilde{u}_{n_i}^{(k_i)})^*-w_{n_i}^{(k_i)}\|+\|\tilde{u}_{n_i}^{(k_i)}(u_{n_i}^{(k_i)})^*-(w_{n_i}^{(k_i)})^*)\|\\
        &<\tfrac{4}{i}\xrightarrow{i\to \infty}0.
    }
    Therefore, the Claim follows. This shows that $\tilde{P}$ is open, whence $P(V_A)=V_{A/I}$. 
\end{proof}

\subsection{The general case}\label{subsec nonseparable}
As we have seen, Theorem \ref{thm:AFconsequence} follows from Theorem \ref{thm:main}. We now partially extend Theorem \ref{thm:AFconsequence} to (not necessarily separable) locally AF algebras.


To ease the notation, we follow the convention in Farah--Katsura \cite{MR2673735FarahKatsura1}.  For a $C^*$-algebra $A$, $x\in A$, $S\subset A$ and $\varepsilon>0$, we write $x\in_{\varepsilon}S$ if there exists $y\in S$ such that $\|x-y\|<\varepsilon$ holds, and for $S_1,S_2\subset A$, we write $S_1\subseteq_{\varepsilon}S_2$ if $x\in_{\varepsilon}S_2$ for every $x\in S_1$. 
Recall that a $C^*$-algebra $A$ is called locally AF (or locally finite-dimensional, LF), if for every finite subset $F$ of $A$ and $\varepsilon>0$, there exists a finite-dimensional subalgebra $D$ of $A$ such that $F\subseteq_{\varepsilon}D$. By \cite[Theorem 2.2]{BratteliMR0312282}, any separable locally AF algebra is an AF algebra and vice versa. (actually, by \cite[Theorem 1.5]{MR2673735FarahKatsura1}, any locally AF $C^*$-algebra of character density $\le \aleph_1$ is AF and for any $\kappa>\aleph_1$, there is a locally AF algebra of character density $\kappa$ which is not AF).

\begin{theorem}\label{thm: main locally AF}
Let $A$ be a unital locally AF algebra, $I\subseteq A$ be a proper closed two-sided ideal. Then the canonical quotient map $P:V_A\to V_{A/I}$ satisfies that 
\begin{itemize}\item[(1)]\label{it main Loc AF-1} $\mathrm{Ker}(P)=V_I$ and consequently, $P$ naturally factorizes as $\tilde P\circ Q$, where $Q:V_A\to V_A/V_I$ is the quotient map; 
\item[(2)]\label{it main Loc AF-2} $\tilde P:V_A/V_I\to V_{A/I}$ is a topological isomorphism onto image.
\end{itemize}
\end{theorem}

In order to reduce the argument to the separable setting, we will need the following result.

\begin{proposition}\label{prop: downwardLS}
Let $A$ be a unital (not necessarily separable) locally AF algebra, $I$ be a closed proper two-sided ideal of $A$, and let $\phi\in V_A$. Then there exists an increasing net $(A_j)_{j\in J}$ of unital separable {\rm C}$^*$ subalgebras such that $A=\bigcup_{j\in J}A_j$ and for each $j\in J$,  the following conditions hold. 
\begin{list}{}{}
\item[{\rm (i)}] 
$\phi(A_j)=A_j$ and $\phi|_{A_j}\in V_{A_j}$. 
\item[{\rm (ii)}] $A_j$ is an AF algebra. 

\end{list}

\end{proposition}

The proof is a combination of the next two results.

\begin{lemma}\label{lem increasing net of inv sep subalgs}
    Let $A$ be a (not necessarily separable) unital {\rm C}$^*$-algebra, $I$ be a proper closed two-sided ideal of $A$. Let $\phi\in V_A$. 
    Then there exists an increasing net of unital separable {\rm C}$^*$-subalgebras of $A$ such that $\phi(A_j)=A_j,\,\phi|_{A_j}\in V_{A_j}$ for every $j\in J$, and $A=\bigcup_{j\in J}A_j$. The net is cofinal in the net of all separable unital {\rm C}$^*$-subalgebras of $A$. 
\end{lemma}
\begin{proof}
For $\phi\in V_A$, let $\mathcal{A}_\phi$ be the set of all separable subalgebras $B$ of $A$ such that $\phi[B]=B$ and $\phi\upharpoonright B\in V_B$. It suffices to show that $\mathcal{A}_\phi$ is a cofinal subnet in the net of all separable unital ${\rm C}^*$-subalgebras of $A$.

Let $B_0\subseteq A$ be a separable subalgebra and let $D=\{d_i\mid i\in\N\}\subseteq B_0$ be a countable dense subset. We may recursively choose a sequence $u_n\in U_A$ of the form 
    \begin{equation}
        u_n=e^{\ri\, b_{n,1}}\cdots e^{\ri\, b_{n,t_n}}\label{eq u_nj}
    \end{equation}
    where $b_{n,t}\,(1\le t\le t_n)$ are self-adjoint elements in $A$, such that $\|\phi(x)-u_n x u_n^*\|<\frac{1}{n}$ for all $x$ in the set
    \[\{\phi^k(d_i)\mid 0\le |k|\le n, i\leq n\}\cup \{b_{m,t}\mid 1\le m\le n-1, 1\le t\le t_m\}.\]

    Then define $B\supseteq B_0$ to be the unital C$^*$-algebra generated by the set 
    \[\{\phi^k(d)\mid k\in \Z, d\in D\}\cup \{b_{n,t}\mid n\in \N,t\le t_n\}.\]
    Then $u_n\in U_B$, $B$ is a unital separable $\phi$-invariant C$^*$-subalgebra of $A$, and 
    \[\phi(x)=\lim_{n\to \infty}{\rm Ad}(u_n)(x),\,\,\,x\in B.\]
\end{proof}

We will also need the next lemma from \cite[Lemma 2.17]{MR2673735FarahKatsura1}, which is a downward L\" owenheim--Skolem-type theorem for metric structures. We include the proof for completeness.

\begin{lemma}\cite{MR2673735FarahKatsura1}\label{lem locAF axiomatizable}
Let $A$ be a unital locally AF algebra, $B_0$ be a unital separable $C^*$-subalgebra of $A$. Then there exists a unital separable AF subalgebra $B$ of $A$ containing. $B_0$. 
\end{lemma}

\begin{proof}
We first show that for any unital separable ${\rm C}^*$-subalgebra $C$ of $A$, there exist a unital separable ${\rm C}^*$-subalgebra $\widetilde{C}$ of $A$ containing $C$ such that for every $\varepsilon>0$ and every finite set $\mathcal{F}\subset C$, there exists a finite-dimensional unital $C^*$-subalgebra $D$ of $\widetilde{C}$ such that $\mathcal{F}\subseteq_{\varepsilon}D$. Let $S=\{x_1,x_2,\dots\}$ be a countable dense subset of $C$. For each $n\in \N$, use the local AF property of $A$ to find a finite-dimensional $C^*$-subalgebra $D_n$ of $A$ such that $\{x_1,\dots,x_n\}\subseteq_{1/n}D_n$. Let $\widetilde{C}=C^*(S,D_1,D_2,\dots)$. Then $\widetilde{C}$ has the required property. 

Then by induction, we construct an increasing sequence $B_0\subseteq B_1\subseteq B_2\subseteq \dots \subseteq A$ of unital separable {\rm C}$^*$-subalgebras of $A$ such that for each $n\ge 0$, finite set $\mathcal{F}\subset B_n$ and for every $\varepsilon>0$, there exists a finite-dimensional $C^*$-subalgebra $D$ of $B_{n+1}$ such that $\mathcal{F}\subseteq_{\varepsilon}D$. Let $B=\overline{\bigcup_{n=0}^{\infty}B_n}$, which is a unital separable $C^*$-subalgebra of $A$. We show that $B$ is locally AF, hence an AF algebra. 
Let $\varepsilon>0$ and $\mathcal{F}\subset B$ be a finite set. Then there exists $n\in \N$ such that $\mathcal{F}\subseteq_{\varepsilon/2}\tilde{\mathcal{F}}$ for some finite set $\tilde{\mathcal{F}}\subset B_n$. 
There exists a finite-dimensional ${\rm C}^*$-subalgebra $D$ of $B_{n+1}\subset B$ such that $\tilde{\mathcal{F}}\subseteq_{\varepsilon/2}D$. 
Therefore $\mathcal{F}\subseteq_{\varepsilon}D$. This shows that $B$ is locally AF. 
\end{proof}

\begin{proof}[Proof of Proposition \ref{prop: downwardLS}]

Let $B_0$ be a unital separable C$^*$-subalgebra of $A$. By a repeated use of Lemmas~\ref{lem increasing net of inv sep subalgs} and~\ref{lem locAF axiomatizable}, there exist unital separable C$^*$-subalgebras $B_0\subset A_1\subset B_1\subset A_2\subset B_2\subset \cdots$ such that for every $k\in \N$, the following conditions hold:
(i) $\phi(A_k)=A_k$, $\phi|_{A_k}\in V_{A_k}$; (ii) $B_k$ is an AF algebra.

Then it is straightforward to show that the unital separable $C^*$-subalgebra $B$ of $A$ defined by 
$$B:=\overline{\bigcup_{n\in \N}A_n}=\overline{\bigcup_{n\in \N}B_n}$$
satisfies $B_0\subset B$ and the following two conditions hold: (i) $\phi(B)=B,\,\phi|_B\in V_B$;  (ii) $B$ is an AF algebra.  

This shows that the net $\mathcal{A}'_{\phi}$ of all unital separable C$^*$-subalgebras of $A$ satisfying the conditions (i), (ii), in the statement of Proposition \ref{prop: downwardLS} is cofinal in the increasing net of unital separable C$^*$-subalgebras of $A$. This finishes the proof. 
\end{proof}

\begin{proof}[Proof of Theorem~\ref{thm: main locally AF} (i)]
    It is enough to show that ${\rm Ker}(P)\subset V_I$. Let $\phi\in {\rm Ker}(P)$. By Proposition \ref{prop: downwardLS}, there exists an increasing net $(A_j)_{j\in J}$ of globally $\phi$-invariant, separable unital AF subalgebras of $A$ such that the conditions $A=\bigcup_{j\in J}A_j,\, \phi(A_j)=A_j$ and $\phi|_{A_j}\in V_{A_j}$ hold.
   
    Note that because $\phi\in V_A$, we have $\phi(I)=I$ as well. Note also that for each $j\in J$, $I_j=A_j\cap I$ (resp. $A_j/I_j$) is a closed two-sided ideal in (resp. a quotient of) a separable unital AF algebra, whence it is AF as well. In particular, both $U((A_j/I_j)'\cap (A_j/I_j)_{\omega})$ and $U(\widetilde{I}_j)$ are connected (here we used Proposition \ref{prop F(AF) has connected unitary group}).
    Let $a\in A$. Then $a\in A_j$ for some $j\in J$. Then $\phi(a)-a\in I$, and by $A_j/(A_j\cap I)\cong (A_j+I)/I$ (see e.g. \cite[p. 85]{Davidson-book}), we have $0={\rm dist}(\phi(a)-a,I)={\rm dist}(\phi(a)-a,A_j\cap I)$. That is, we have $\phi(a)+I_j=a+I_j$. If we denote by $P_j$ the quotient map $V_{A_j}\to V_{A_j/I_j}$, then $\phi|_{A_j}\in {\rm Ker}(P_j)$. 
    By Theorem \ref{thm:AFconsequence}, we have $\phi|_{A_j}\in V_{I_j}$. 

    Finally, let $F\subset A$ be a nonempty finite subset, and $\varepsilon>0$. By $A=\bigcup_{j\in J}A_j$, there exists $j\in J$ such that $F\subset A_j$. 
    By $\phi|_{A_j}\in V_{I_j}$, there exists $u\in U_{I_j}$ such that 
    $$\max_{a\in F}\|\phi(a)-uau^*\|<\varepsilon.$$
    Since $F$ and $\varepsilon$ are arbitrary and $u\in U_I$, it follows that $\phi$ is in the closure of $\{{\rm Ad}(u)\mid u\in U_I\}$ in the point-norm topology. Thus $\phi\in V_I$ holds.
\end{proof}

In order to avoid any misunderstanding, we recall a definition of a subnet that we use.
\begin{definition}
A subnet of a net $(x_i)_{i\in I}$ in a topological space $X$ is a net of the form $(y_{\alpha})_{\alpha\in A}$, where there exists a map $h\colon A\to I$ with the property that for every $i\in I$, there exists $\alpha_0\in A$ such that $h(\alpha)\ge i$ for every $\alpha \ge \alpha_0$, and \[y_{\alpha}=x_{h(\alpha)},\,\alpha\in A.\]
This definition is more general than the subnet in the sense of Willard, where it is required that the reindexing map $h$ is order-preserving and has cofinal image in $I$. 
\end{definition}
The proof of the next lemma is standard, so we omit the proof. 
\begin{lemma}\label{lem subnet}
Let $X$ be a topological space, $(x_i)_{i\in I}$ be a net in $X$, and let $x\in X$. If every subnet of $(x_i)_{i\in I}$ has a further subnet converging to $x$, then $(x_i)_{i\in I}$ itself converges to $x$. 
\end{lemma}

\begin{proof}[Proof of Theorem~\ref{thm: main locally AF} (ii)]
    
It suffices to show that $\tilde P$ is a homeomorphism with its image.
Suppose $(\phi_{\lambda})_{\lambda\in \Lambda}$ is a net in $V_A$ such that $P(\phi_{\lambda})\to {\rm id}$ in $V_{A/I}$. 
We show that $[\phi_{\lambda}]\to [{\rm id}]$ in $V_A/V_I$. 
To this purpose let $([\phi_{\lambda'}])_{\lambda'\in \Lambda'}$ be a subnet of $([\phi_{\lambda}])_{\lambda\in \Lambda}$. If we show that it admits a further subnet convergent to $[{\rm id}]$, then we obtain $[\phi_{\lambda}]\to {\rm id}$ by Lemma \ref{lem subnet}. Thus, we may from the beginning assume that $\Lambda=\Lambda'$ to find such a convergent subnet. Moreover, since ${\rm Ad}\,U_A$ is dense in $V_A$, by passing to a further subnet, we may and do assume that for each $\lambda \in \Lambda$, there exist $u_{\lambda}\in U_A$ and $\tilde{\phi}_{\lambda}\in V_A$ such that $\phi_{\lambda}=\tilde{\phi}_{\lambda}\circ {\rm Ad}(u_{\lambda})$ and $\disp \lim_{\lambda}\tilde{\phi}_{\lambda}={\rm id}$ in $V_A$. Indeed, since ${\rm Ad}(U_A)$ is dense in $V_A$, for each $\lambda\in \Lambda$, the set $\{\phi_{\lambda}\circ {\rm Ad}(u)\mid u\in U_A\}$ is dense in $V_A$. Let $(W_i)_{i\in I}$ be an decreasing net of all open neighborhoods of {\rm id} in $V_A$ paritally ordered by reverse inclusion. Define a partial ordering on the set $\widetilde{\Lambda}=\Lambda\times I$ by 
\[(\lambda,i)\le (\lambda',i')\iff \lambda\le \lambda',\,\,i\le i'\]
for $\lambda,\lambda'\in \Lambda$ and $i,i'\in I$. Then for each $\tilde{\lambda}=(\lambda,i)\in \widetilde{\Lambda}$, there exists $u_{\tilde{\lambda}}\in U_A$ such that $\phi_{\lambda}\circ {\rm Ad}(u_{\tilde{\lambda}}^*)\in W_i$. 
Define $h\colon \widetilde{\Lambda}\ni (\lambda,i)\mapsto \lambda\in \Lambda$. Then 
$(\phi_{\tilde{\lambda}})_{\tilde{\lambda}\in \widetilde{\Lambda}}$ is a subnet of $(\phi_{\lambda})_{\lambda\in \Lambda}$ with reindexing map $h$, and $\lim_{\tilde{\lambda}}\phi_{\tilde{\lambda}}\circ {\rm Ad}(u_{\tilde{\lambda}}^*)={\rm id}$ in $V_A$. Thus $\tilde{\phi}_{\tilde{\lambda}}=\phi_{\lambda}\circ {\rm Ad}(u_{\tilde{\lambda}}^*)$ works.

Then in $V_A/V_I$, $[\phi_{\lambda}]=[\tilde{\phi}_{\lambda}][{\rm Ad}(u_{\lambda})]\to [{\rm id}]$ if and only if $[{\rm Ad}(u_{\lambda})]\to [{\rm id}]$. Thus, we may assume that $\phi_{\lambda}={\rm Ad}(u_{\lambda})\,(\lambda\in \Lambda)$. 
 
Consider the partially ordered set $K=\{(F,\varepsilon)\mid F\subset A,\,\sharp F<\infty,\,\varepsilon>0\}$ with the partial ordering given by 
\[(F_1,\varepsilon_1)\le (F_2,\varepsilon_2)\iff F_1\subset F_2,\,\,\varepsilon_1>\varepsilon_2.\]
By $P(\phi_{\lambda})\to {\rm id}$, for each $k=(F,\varepsilon)\in K$, there exists $\lambda_k\in \Lambda$ such that for every $\lambda\ge \lambda_k$ in $\Lambda$, the following inequality holds. 
\[{\rm dist}(u_{\lambda}au_{\lambda}^*-a,I)<\varepsilon,\,\,a\in F.\]

Then $([\phi_{\lambda_k}])_{k\in K}$ is a net  satisfying $P(\phi_{\lambda_k})\to {\rm id}$ in $V_{A/I}$. It may not be a subnet of $([\phi_{\lambda}])_{\lambda\in \Lambda}$ because $\{\lambda_k\mid k\in K\}$ may not be cofinal in $\Lambda$. However, we can always enlarge the index set $K$ to make it a subnet. More precisely, define $\tilde{K}=K\times \Lambda$ and define a partial ordering on it by  
\[(k_1,\lambda_1)\le (k_2,\lambda_2)\iff k_1\le k_2\text{\,\,and\,\,}\lambda_1\le \lambda_2\]
for $k_1,k_2\in K$ and $\lambda_1,\lambda_2\in \Lambda$. Since $\Lambda$ is cofinal, for each $\tilde{k}=(k,\lambda)\in \tilde{K}$ there exists $\lambda_{\tilde{k}}\ge \lambda_k,\lambda$. Then $([\phi_{\lambda_{\tilde{k}}}])_{\tilde{k}\in \tilde{K}}$ is a subnet of $([\phi_{\lambda}])_{\lambda\in \Lambda}$ and  we still have  $P(\phi_{\lambda_{\tilde{k}}})\to {\rm id}$ in $V_{A/I}$. With this in mind, to ease the notation, we shall treat $([\phi_{\lambda_k}])_{k\in K}$ as if it is a subnet of $([\phi_{\lambda}])_{\lambda\in \Lambda})$. 
Thus we may and do assume that $\Lambda=K$ (thus $\phi_{\lambda_k}$ will be denoted by $\phi_{k}$ to ease the notation). Moreover, for each $k\in K$ there exist self-adjoint elements $b_{k,1},\dots, b_{k,t_k}$ such that 
\[u_k=e^{\ri b_{k,1}}\cdots e^{\ri b_{k,t_k}},\,\,k\in K.\]
We then proceed as follows. 
We fix $k=(F,\varepsilon)\in K$. Let $B_0$ be the unital separable $C^*$-subalgebra of $A$ generated by $F$, and choose a countable dense subset $\{d_j\}_{k\in \N}$ of $B_0$. 
We then construct, by induction on $j\in \N$, a sequence $\kappa_0=k\le \kappa_1\le \kappa_2\le \dots$ in $K$ with self-adjoint elements $b_{\kappa_j,1},\dots, b_{\kappa_j,t_j}\in A$ satisfying $u_{\kappa_j}=e^{\ri b_{\kappa_j,1}}e^{\ri b_{\kappa_j,t_j}}$, such that for each $j\ge 1$, the inequality 
\[{\rm dist}(u_{\kappa_{j+1}}au_{\kappa_{j+1}}^*-a,I)<\tfrac{\varepsilon}{j}\]
holds for every $a$ in the finite set
\[\left \{\phi_{\kappa_j}^{\ell}(c)\,\middle|\, 0\le |\ell|\le j,\,c\in \{d_1,\dots,d_j\}\cup \{b_{\kappa_{j'},t}\mid 0\le j'\le j,1\le t\le t_{j'}\}\right \}.\]
Let $\tilde B_1$ be the unital separable $C^*$-subalgebra of $A$ generated by the countable set
\[\left \{\phi_{\kappa_j}^{\ell}(c)\,\middle|\, \ell \in \Z, j\in \N,\,c\in \{d_j\}_{j\in \N}\cup \{b_{\kappa_{j},t}\mid j\ge 0,\,1\le t\le t_{j}\}\right \}.\]

Then $u_{\kappa_j}\in U_{\tilde B_1}$ for every $j\ge 0$ and we have 
\[{\rm dist}(u_{\kappa_j}au_{\kappa_j}^*-a,I)={\rm dist}(u_{\kappa_j}au_{\kappa_j}^*-a,\tilde B_1\cap I)\xrightarrow{j\to \infty}0\]
for every $a\in \tilde B_1$. Here, we used the fact that $(\tilde B_1+I)/I\cong \tilde B_1/(\tilde B_1\cap I)$. Therefore, we obtain
\[{\rm Ad}(u_{\kappa_j})\to {\rm id}\,\,{\rm in\,} V_{\tilde B_1/(\tilde B_1\cap I)}.\]

Next by Lemma~\ref{lem locAF axiomatizable}, there exists a unital separable AF subalgebra $B_1$ of $A$ containing $\tilde B_1$. We now fix $k\leq k_1=(F_1,\varepsilon_1)$ and apply the same argument as for $B_0$ and $k$ to $B_1$ and $k_1$ to obtain a sequence $\kappa_0=k\leq \kappa_1=k_1\leq \kappa_2\leq \kappa_3\leq\ldots$ in $K$ and a unital separable $C^*$-subalgebra $\tilde B_2$ containing $B_1$ such that \[{\rm Ad}(u_{\kappa_j})\to {\rm id}\,\,{\rm in\,} V_{\tilde B_2/(\tilde B_2\cap I)}.\] We then again apply Lemma~\ref{lem locAF axiomatizable} to get a unital separable AF subalgebra $B_2$ of $A$ containing $\tilde B_2$ and we fix new $k_1\leq k_2=(F_2,\varepsilon_2)$. We continue in the same way to get an increasing sequence of unital separable subalgebras \[B_0\subseteq \tilde B_1\subseteq\ldots\subseteq \tilde B_n\subseteq B_n\subseteq\] and increasing sequence $k=k_0\leq k_1\leq\ldots \leq k_n=(F_n,\varepsilon_n)\leq\ldots$, where for each $n\in\N$, $B_n$ is AF, $\varepsilon_n\to 0$, and by a careful bookkeeping \[B:=\overline{\bigcup_{n\in\N} F_n}=\overline{\bigcup_{n\in\N} B_n} .\] Indeed, the equality $\overline{\bigcup_{n\in\N} F_n}=\overline{\bigcup_{n\in\N} B_n}$ can be guaranteed e.g. by choosing $F_n$, for each $n\in\N$, so that it contains $\{d_{k,i}\mid  k,i<n\}$, where for each $k\in\N$, $\{d_{k,i}\mid i\in\N\}$ is a countable dense set in $B_k$ which is fixed as soon as $B_k$ is defined.

It follows that $B$ is a separable unital AF algebra. Moreover, we have $u_{k_j}\in U_B$ for every $j\ge 0$ and
\[{\rm dist}(u_{k_j}au_{k_j}^*-a,I)={\rm dist}(u_{k_j}au_{k_j}^*-a,\tilde B\cap I)\xrightarrow{j\to \infty}0\]
for every $a\in \tilde B$.

Thus as above, we get \[{\rm Ad}(u_{k_j})\to {\rm id}\,\,{\rm in\,} V_{B/(B\cap I)}.\]

Since $B,B\cap I$ and $B/(B\cap I)$ are separable AF algebras, by Proposition \ref{prop F(AF) has connected unitary group}, we may apply the same argument in the proof of ${\rm Ker}(P)=V_I$ part of Theorem \ref{thm:main} to obtain a further subsequence $k_{j_1}\le k_{j_2}\le \dots $, unitaries $(\tilde{u}_{k_{j_i}})_{i\in \N}$ in $U_B$ and $(w_{k_{j_i}})_{i\in \N}$ in $U_{I\cap B}$ such that 
\[\lim_{i\to \infty}\|u_{k_{j_i}}\tilde{u}_{k_{j_i}}^*-w_{k_{j_i}}\|=0.\]
Then 
\[\lim_{i\to \infty}{\rm Ad}(u_{k_{j_i}})(b)=\lim_{i\to \infty}{\rm Ad}(w_{k_{j_i}})(b),\,\,\,b\in B.\]
This implies that for a given fixed $k=(F,\varepsilon)\in K$, there exists $k'\ge k$ ($k'=k_{j_i}$ for sufficiently large $i$) and $w_{k'}\in U_{I}$ such that 
\[\|u_{k'}au_{k'}^*-w_{k'}aw_{k'}^*\|<\varepsilon,\,\,\,a\in F.\]
Then we have $\|w_{k'}^*u_{k'}au^*_{k'}w_{k'}-a\|<\varepsilon$. Applying this to every $k\in K$ we obtain as a result a subnet $(\phi_{k'})_{k'\in K'}$ on a directed subset $K'\subset K$, for which we have 
\[[\phi_{k'}]=[{\rm Ad}(w_{k'}^*u_{k'})]\to [{\rm id}]\,{\rm in\,}V_A/V_I.\]
Thus, we have shown that any subnet of $([\phi_k])_{k\in K}$ contains a further subnet converging to id in $V_A/V_I$, whence $[\phi_k]\to [{\rm id}]$ in $V_A/V_I$. This shows that $V_A/V_I\cong P(V_A)$. 
\end{proof}

We recall that a topological group $G$ is called \emph{Raikov-complete} if it is complete with respect to the two-sided uniformity. That is, every Cauchy net (with respect to both left and right uniformities) converges; see \cite[Section 3.6]{ArTkbook} and \cite[Theorem 3.6.25]{ArTkbook}.

It is unclear to us whether $V_A/V_I$ is a Raikov-complete topological group (clearly both $V_A$ and $V_I$ are Raikov-complete). 
\begin{question}\label{question:Raikov}
Let $A$ be a (not necessarily separable) unital $C^*$-algebra, $I$ be a proper two-sided closed ideal of $A$. Is  $V_A/V_I$ Raikov-complete?
\end{question}
An affirmative answer would guarantee that the induced map $\tilde{P}\colon V_A/V_I\to V_{A/I}$ is onto, for any unital locally AF algebra $A$, whence a topological group isomorphism. Indeed, then the identity map from the dense Raikov-complete subgroup $\tilde P[V_A/V_I]\subseteq V_{A/I}$ to itself would extend to a continuous homomorphism from $V_{A/I}$ onto $P[V_A/V_I]$ by \cite[Proposition 3.6.12]{ArTkbook}, showing that $V_{A/I}=\tilde P[V_A/V_I]$.

\section{Another remark on tensor products of ucp maps}\label{sec Kirchberg central sequence}
Finally, we would like to visit yet another theorem by Kirchberg on the central sequence algebra and point out that it leads to an observation regarding the tensor product of unital completely positive (ucp) maps between von Neumann algebras. Recall that a separable purely infinite simple C$^*$-algebra is called a Kirchberg algebra. A remarkable property of Kirchberg algebras discovered by Kirchberg himself is that its central sequence algebra is simple. In fact, he showed the following important theorem (we only state the result for the unital case). 
\begin{theorem}[Kirchberg \cite{Kirchberg3draft,kirchbergAbelMR2265050,KirchbergPhillips2000MR1745197}]\label{thm cs pi simple}
Let $A$ be a unital Kichberg algebra and $\omega$ be a free ultrafilter on $\N$. Then the central sequence algebra $A'\cap A_{\omega}$ is simple and purely infinite. Conversely, if $A$ is a unital separable {\rm C}$^*$-algebra for which $A'\cap A_{\omega}$ is simple and non-trivial, then $A$ is simple, purely infinite and nuclear (thus it is a Kirchberg algebra).   \end{theorem}
See \cite[Theorem 2.12]{kirchbergAbelMR2265050}
and \cite[Proposition 3.4]{KirchbergPhillips2000MR1745197} for proofs. The theorem is used in the Kirchberg--Phillips' classification \cite{Kirchberg3draft,KirchbergPhillips2000MR1745197} of UCT Kirchberg algebras. 
Recall that a von Neumann algebra $M$ on a Hilbert space $H$ is called injective if there exists a norm one projection $E\colon \mathbb{B}(H)\to M$. This property was originally introduced and called the extension property by Hakeda--Tomiyama in \cite{HakedaTomiyamaMR222656}. By Connes' theorem \cite{ConnesMR454659}, the injectivity is equivalent to the hyperfiniteness. Thus, the tensor product $M_1\overline{\otimes}M_2$ of injective von Neumann algebras $M_i\subset \mathbb{B}(H_i)\,(i=1,2)$ is again injective. In fact it is possible, as was shown by Tomiyama \cite[Theorem 4]{Tomiyama1969MR246141} (as he remarks, the arguments were essentially present already in \cite[Lemma 2.3, Theorem 3.2]{HakedaTomiyamaMR222656}), that even without assuming the normality of $E_1$ nor $E_2$, the tensor product $E_1\odot E_2\colon \mathbb{B}(H_1)\odot \mathbb{B}(H_2)\to M_1\odot M_2$ (here, $\odot$ denotes the algebraic tensor product) extends to a norm one projection $E_1\otimes E_2\colon \mathbb{B}(H_1)\overline{\otimes}\mathbb{B}(H_2)\to M_1\overline{\otimes}M_2$. Then it was further generalized by Nagisa--Tomiyama \cite{NagisaTomiyamaMR620290} that for any (not necessarily normal) ucp maps $\theta_i\colon M_i\to N_i\,(i=1,2)$ between von Neumann algebras, the tensor product $\theta_1\odot \theta_2\colon M_1\odot M_2\to N_1\odot N_2$ extends to a ucp map $\theta \colon M_1\overline{\otimes}M_2\to N_1\overline{\otimes}N_2$.
Nevertheless, we show that even all von Neumann algebras involved are atomic and both of $\theta_i$ are $*$-homomorphisms, the extension $\theta$ (as a completely positive map) may fail to be a $*$-homomorphism (it is a $*$-homomorphism when restricted to $M_1\otimes_{\rm min}N_1$). This can be seen e.g., as follows, thanks to Kirchberg's above theorem.    
\begin{example}
Fix a free ultrafilter $\omega$ on $\mathbb{N}$. Let $\theta_{\omega}\colon \ell^{\infty}(\mathbb{N})\ni (a_n)_{n=1}^{\infty}\mapsto \lim_{n\to \omega}a_n\in \mathbb{C}$ be the associated character. We show that the completely positive map $\text{id}_{\mathbb{B}(\ell^2)}\odot \theta_{\omega}\colon \mathbb{B}(\ell^2(\mathbb{N}))\odot \ell^{\infty}(\mathbb{N})\to \mathbb{B}(\ell^2(\mathbb{N}))\otimes \mathbb{C}=\mathbb{B}(\ell^2(\mathbb{N}))$ cannot be extended to a $*$-homomorphism  $\mathbb{B}(\ell^2(\mathbb{N}))\overline{\otimes}\ell^{\infty}(\mathbb{N})\to \mathbb{B}(\ell^2(\mathbb{N}))$.
\end{example}
\begin{proof} Write $\ell^p(\N)=\ell^p$. 
Assume by contradiction that there exists a $*$-homomorphism
\[\Phi\colon \mathbb{B}(\ell^2)\overline{\otimes}\ell^{\infty}\to \mathbb{B}(\ell^2)\] extending $\text{id}_{\mathbb{B}(\ell^2)}\odot \theta_{\omega}$. We identify $\mathbb{B}(\ell^2)\overline{\otimes}\ell^{\infty}=\ell^{\infty}(\mathbb{N},\mathbb{B}(\ell^2))$, the C$^*$-algebra of all bounded sequences of operators in $\mathbb{B}(\ell^2)$. 
Let $\pi\colon \mathcal{O}_2\to \mathbb{B}(\ell^2)$ be an irreducible representation of the Cuntz algebra $\mathcal{O}_2$. Since $\mathcal{O}_2$ is simple, $\pi (\mathcal{O}_2)$ is $*$-isomorphic to $\mathcal{O}_2$. Let $c_{\omega}(\mathbb{B}(\ell^2))$ be the C$^*$-subalgebra of $\ell^{\infty}(\mathbb{N},\mathbb{B}(\ell^2))$ consisting of all bounded sequences $(x_n)_{n=1}^{\infty}$ in $\mathbb{B}(\ell^2)$ such that $\lim_{n\to \omega}\|x_n\|=0$. We first show that $c_{\omega}(\mathbb{B}(\ell^2))\subset \text{Ker}(\Phi)$. Let $x=(x_n)_{n=1}^{\infty}\in c_{\omega}(\mathbb{B}(\ell^2))$. Note that $x$ corresponds to $\sum_{n=1}^{\infty}x_n\otimes \delta_n\in \mathbb{B}(\ell^2)\overline{\otimes}\ell^{\infty}$, where $\delta_n\in \ell^{\infty}$ is the element given by $\delta_n(m)=\begin{cases}\ \ \ 1 & (m=n)\\
 \ \ \ 0 & (m\neq n)\end{cases}$. 
It follows that 
 \eqa{
\Phi(x)&=\Phi\left (\sum_{n=1}^{\infty}x_n\otimes \delta_n\right )=\Phi\left (\sum_{n=1}^{\infty}\frac{x_n}{\|x_n\|+\tfrac{1}{n}}\otimes \delta_n\cdot \sum_{n=1}^{\infty}1\otimes (\|x_n\|+\tfrac{1}{n})\delta_n\right )\\
&=\Phi\left (\sum_{n=1}^{\infty}\frac{x_n}{\|x_n\|+\tfrac{1}{n}}\otimes \delta_n\right )\Phi\left (\sum_{n=1}^{\infty}1\otimes (\|x_n\|+\tfrac{1}{n})\delta_n\right )\\
&=\Phi\left (\sum_{n=1}^{\infty}\frac{x_n}{\|x_n\|+\tfrac{1}{n}}\otimes \delta_n\right )\lim_{n\to \omega}(\|x_n\|+\tfrac{1}{n})\\
&=0. 
}
Therefore, $x\in \text{Ker}(\Phi)$ and $c_{\omega}(\mathbb{B}(\ell^2))\subset \text{Ker}(\Phi)$ holds.\\ 
Then there exists a $*$-homomorphism $\overline{\Phi}\colon \mathbb{B}(\ell^2)_{\omega}:=\dfrac{\ell^{\infty}(\mathbb{N},\mathbb{B}(\ell^2))}{c_{\omega}(\mathbb{B}(\ell^2))}\to \mathbb{B}(\ell^2)$ such that $\Phi=\overline{\Phi}\circ q_{\omega}$, where $q_{\omega}\colon \ell^{\infty}(\mathbb{N},\mathbb{B}(\ell^2))\to \mathbb{B}(\ell^2)_{\omega}$ is the canonical surjection. Note that $\pi(\mathcal{O}_2)_{\omega}:=\dfrac{\ell^{\infty}(\mathbb{N},\pi(\mathcal{O}_2))}{c_{\omega}(\pi(\mathcal{O}_2))}$ can be identified with a $*$-subalgebra of $\mathbb{B}(\ell^2)_{\omega}$. Then, regarding $\pi(\mathcal{O}_2)\subset \pi(\mathcal{O}_2)_{\omega}$ as usual by diagonal embedding, we define $\Psi:=\overline{\Phi}|_{\pi(\mathcal{O}_2)'\cap \pi(\mathcal{O}_2)_{\omega}}\colon \pi(\mathcal{O}_2)'\cap \pi(\mathcal{O}_2)_{\omega}\to \mathbb{B}(\ell^2)$. Note that for $a\in \pi(\mathcal{O}_2)$ and $x\in \pi(\mathcal{O}_2)'\cap \pi(\mathcal{O}_2)_{\omega}$, we have (use $\overline{\Phi}|_{\pi(\mathcal{O}_2)}=\text{id}_{\pi(\mathcal{O}_2)}$)
\eqa{
\overline{\Phi}(ax)&=\overline{\Phi}(a)\overline{\Phi}(x)=a\Psi (x),\\
\overline{\Phi}(xa)&=\overline{\Phi}(x)\overline{\Phi}(a)=\Psi (x)a,
}
whence $a\Psi(x)=\Psi(x)a$ by $ax=xa$. Since $a\in \pi(\mathcal{O}_2)$ is arbitrary and $\pi$ is irreducible, we get $\Psi(x)\in \pi(\mathcal{O}_2)'\cap \mathbb{B}(\ell^2)=\mathbb{C}$. This shows that $\Psi\colon \pi(\mathcal{O}_2)'\cap \pi(\mathcal{O}_2)_{\omega}\to \mathbb{C}$ is a nonzero character. However, $\pi(\mathcal{O}_2)\cong \mathcal{O}_2$ is a Kirchberg algebra, so that by Theorem \ref{thm cs pi simple}, the central sequence algebra $\pi(\mathcal{O}_2)'\cap \pi(\mathcal{O}_2)_{\omega}$ is purely infinite and simple, a contradiction. 
\end{proof}

\appendix 
\section{Proof of Kirchberg's lemma (Proposition \ref{prop kirchberg lifting of cs})}
Before the proof, recall that a positive element $h$ in a (not necessarily unital) {\rm C}$^*$-algebra $A$ is called strictly positive, if $\overline{hAh}=A$ (equivalently, if $\overline{hA}=A$). $A$ admits a strictly positive element if and only if $A$ is $\sigma$-unital, meaning that $A$ admits a countable approximate unit. In particular, any separable C$^*$-algebra admits a strictly positive element. 
\begin{proof}
    (i) 
It is clear that $\iota_{\omega}$ is injective, ${\rm Im}(\iota_{\omega})(J_{\omega})\subset {\rm Ker}((\pi_J)_{\omega})$ and $(\pi_J)_{\omega}$ is onto. Therefore it suffices to show that ${\rm Ker}((\pi_J)_{\omega})\subset {\rm Im}(\iota_{\omega})$. Let $a=(a_1,a_2,\dots)+c_{\omega}(A)$ be in the kernel of $(\pi_J)_{\omega}$. Then $\disp \lim_{n\to \omega}\|\pi_J(a_n)\|=0$. 
Choose for each $n\in \N$ an element $a_n'\in A$ such that $\|a_n'\|=\|\pi_J(a_n)\|$ and $\pi_J(a_n')=\pi_J(a_n)$. Then $(a_n')_{n=1}^{\infty}\in c_{\omega}(A)$ and $c_n\colon =a_n-a_n'\in J$ for every $n\in \N$. Thus $c=(c_1,c_2,\dots)+c_{\omega}(J)$ satisfies $\iota_{\omega}(c)=a$.\\
(ii)\, Assume first that $B\cap J_{\omega}\neq \{0\}$. Let $s$ be a strictly positive contraction in the separable C$^*$-algebra $B\cap J_{\omega}$. Let $t=(t_n)_{n=1}^{\infty}\in \ell^{\infty}(J)_+$ be a seuence such that $s=t+c_{\omega}(J)$. Select a sequence $c_1,c_2,\dots \in B$ which is dense in the closed unit ball of $B$. Let $d_{1,n},d_{2,n},\dots\in A$ be contractions such that 
\[c_n=d_n+c_{\omega}(A),\,d_n=(d_{1,n},d_{2,n},\dots)\in \ell^{\infty}(A),\,\,n\in \N.\]
By considering an approximate identity in $J$ which is quasi-central for $A$, we may find for each $k\in \N$ a positive contraction $f_k\in J$ such that 
\begin{align}
    \|f_kt_k-t_k\|<\frac{1}{k},\\
    \max_{1\le n\le k}\|f_kd_{k,n}-d_{k,n}f_k\|<\frac{1}{k}.
\end{align}
Set $f=(f_1,f_2,\dots)\in \ell^{\infty}(J)$ and 
\[e=f+c_{\omega}(J)=f+c_{\omega}(A),\]
where we view $J_{\omega}\subset A_{\omega}$ via $\iota_{\omega}$. 
$e$ is a positive contraction in $J_{\omega}$. Also $es=s$ and $eb=be$ hold for every $b\in B$. Since $s\in (B\cap J_{\omega})_+$ is strictly positive, we have $\overline{s(B\cap J_{\omega})s}=B\cap J_{\omega}$. Therefore for every $b\in B\cap J_{\omega}$, we have $eb=be=b$.\\
If $B\cap J_{\omega}=\{0\}$, then by the same argument using quasi-central approximate units, we may find a positive contraction $e\in B'\cap J_{\omega}$, and this $e$ does the job.\\
(iii) It is clear that $B'\cap J_{\omega}={\rm Ker}((\pi_J)_{\omega}|_{B'\cap A_{\omega}})$ and $(\pi_J)_{\omega}(B'\cap A_{\omega})\subset (\pi_J)_{\omega}(B)'\cap (A/J)_{\omega}$. Thus if we show that $(\pi_J)_{\omega}(B)'\cap (A/J)_{\omega}\subset (\pi_J)_{\omega}(B'\cap A_{\omega})$ then we would get the exactness of 
\[0\longrightarrow B'\cap J_{\omega}\xrightarrow{\iota_{\omega}} B'\cap A_{\omega}\xrightarrow{(\pi_J)_{\omega}}(\pi_J)_{\omega}(B)'\cap (A/J)_{\omega}\longrightarrow 0.\]
Let $h\in (\pi_J)_{\omega}(B)'\cap (A/J)_{\omega}$ and let $g\in A_{\omega}$ with $(\pi_J)_{\omega}(g)=h$. $D\colon =C^*(B,g)$ is then a separable C$^*$-subalgebra of $A_{\omega}$ and $bg-gb\in D\cap J_{\omega}$ holds for every $b\in B$. By (ii) applied to $D$ in place of $B$, we find a positive contraction $e\in (D'\cap J_{\omega})$ such that $ed=de=d$ for all $d\in D\cap J_{\omega}$.
Set $k=(1-e)g\in A_{\omega}$. Then 
\eqa{
    (\pi_J)_{\omega}(k)&=h-(\pi_J)_{\omega}(e)h\\
    &=h\,\,({\rm by}\,e\in J_{\omega}),
} 
and for all $b\in B$, $gb-bg\in D\cap J_{\omega}$, whence $(1-e)(gb-bg)=0$. Thus
\eqa{
    kb-bk&=(1-e)gb-b(1-e)g\\
    &=(1-e)(gb-bg)+(1-e)bg-b(1-e)g\\
    &=(1-e)bg-b(1-e)g\\
    &=0,
}
because $b\in D$ commutes with $e\in D'\cap J_{\omega}$. Thus $k\in B'\cap A_{\omega}$, which finishes the proof. 

\end{proof}
\section{Iterated Ultralimits}
Let $I,J$ be directed sets, and let $\mathcal{U},\mathcal{V}$ be cofinal ultrafilters on $I$ and $J$, respectively. Then the {\it product ultrafilter}, denoted $\mathcal{U}\otimes \mathcal{V}$ is a filter on $I\times J$ (with the partial ordering $(i,j)\le (i',j')$ if $i\le i'$ and $j\le j'$) given by 
\[\mathcal{U}\otimes \mathcal{V}=\{A\subset I\times J; \{i\in I; \{j\in J; (i,j)\in A\}\in \mathcal{V}\}\in \mathcal{U}\}.\]
The next lemma is well-known in the theory of ultrafilters and can be checked by a straightforward computation. 
\begin{lemma}\label{lem: iterated ultralimit}
$\mathcal{U}\otimes \mathcal{V}$ is a cofinal ultrafilter on $I\times J$.  
Moreover, if $(x_{i,j})_{(i,j)\in I\times J}$ is a doubly indexed sequence in a compact Hausdorff space $X$, then 
\[\lim_{(i,j)\to \mathcal{U}\otimes \mathcal{V}}x_{i,j}=\lim_{i\to \mathcal{U}}\lim_{j\to \mathcal{V}}x_{i,j}.\]
\end{lemma}
\begin{proof} The result is well-known, but for the convenience of the reader we include its proof. 
For $X\subset I\times J$ and $i\in I$, we write $X_i=\{j\in J;\ (i,j)\in X\}$. 
First, we show that $\mathcal{U}\otimes \mathcal{V}$ is a filter on $I\times J$. It is clear that $\emptyset\notin \mathcal{U}\otimes \mathcal{V}$. 
 Let $A,B\subset I\times J$ be such that $A\subset B$ and $A\in \mathcal{U}\otimes \mathcal{V}$. Then for each $i\in I$, $A_i\subset B_i(\subset J)$ and $\{i\in I; A_i\in \mathcal{V}\}\in \mathcal{U}$. This shows that $\{i\in I;\ B_i\in \mathcal{V}\}\in \mathcal{U}$, whence $B\in \mathcal{U}\otimes \mathcal{V}$. Next, let $A,B\in \mathcal{U}\otimes \mathcal{V}$. Then for each $i\in I$, we have $A_i\cap B_i=(A\cap B)_i$, whence 
 \[\{i\in I;A_i\in \mathcal{V}\}\cap \{i\in I;\ B_i\in \mathcal{V}\}\subset \{i\in I;\ (A\cap B)_i\in \mathcal{V}\},\]
 which implies that $\{i\in I;\ (A\cap B)_i\in \mathcal{V}\}\in \mathcal{U}$. Therefore $A\cap B\in \mathcal{U}\otimes \mathcal{V}$. This shows that $\mathcal{U}\otimes \mathcal{V}$ is a filter on $I\times J$.\\ \\
 Next, let $(i_0,j_0)\in I\times J$ and let $S:=\{(i,j)\in I\times J;\ i\ge i_0,\ j\ge j_0\}$. For each $i\in I$, 
$S_i=\begin{cases}\{j\in J;\ j\ge j_0\} & (i\ge i_0)\\ \ \ \ \ \emptyset & (\text{otherwise})\end{cases}$, and if $i\ge i_0$, then $\{j\in J;\ j\ge j_0\}\in \mathcal{V}$ because $\mathcal{V}$ is cofinal. Thus, $\{i\in I;\ S_i\in \mathcal{V}\}=\{i\in I;\ i\ge i_0\}\in \mathcal{U}$ because $\mathcal{U}$ is cofinal. Therefore, $\mathcal{U}\otimes \mathcal{V}$ is cofinal. Finally, let $A\subset I\times J$ be such that $A\notin \mathcal{U}\otimes \mathcal{V}$. 
Then because $\mathcal{U},\mathcal{V}$ are ultrafilters, we have 
\eqa{
\{i\in I;\ A_i\in \mathcal{V}\}\notin \mathcal{U}&\Leftrightarrow \{i\in I;\ A_i\notin \mathcal{V}\}\in \mathcal{U}\\
&\Leftrightarrow \{i\in I;\ (J\setminus A_i)=(I\times J\setminus A)_i\in \mathcal{V}\}\in \mathcal{U},
}
 and the last condition is equivalent to $I\times J\setminus A\in \mathcal{U}\otimes \mathcal{V}$. Therefore, $\mathcal{U}\otimes \mathcal{V}$ is a cofinal ultrafilter on $I\times J$. This finishes the proof of the first assertion. We show the second assertion. 
 Set $x:=\lim_{(i,j)\to \mathcal{U}\otimes \mathcal{V}}x_{i,j}$ and $x_i:=\lim_{j\to \mathcal{V}}x_{i,j}\ (i\in I)$. Let $W$ be an open neighborhood of $x$ in $X$. Since a compact Hausdorff space is regular, there exists an open neighborhood $W_1$ of $x$ such that $x\in W_1\subset \overline{W_1}\subset W$. Then $\{(i,j)\in I\times J;\ x_{i,j}\in W_1\}\in \mathcal{U}\otimes \mathcal{V}$, whence $I_0:=\{i\in I;\ \{j\in J;\ x_{i,j}\in W_1\}\in \mathcal{V}\}\in \mathcal{U}$ holds. Let $i\in I_0$. Then $B:=\{j\in J;\ x_{i,j}\in W_1\}\in \mathcal{V}$. If $V$ is any open neighborhood of $x_i$, then $B':=\{j\in J;\ x_{i,j}\in V\}\in \mathcal{V}$, whence $B\cap B'\in \mathcal{V}$ holds. In particular, we can take $j\in B\cap B'$. Then $x_{i,j}\in V\cap W_1\neq \emptyset$. Since $V$ is arbitrary, this shows that $x_i\in \overline{W_1}\subset W$. Therefore $\mathcal{U}\ni I_0\subset \{i\in I;\ x_i\in W\}$, which shows that $\{i\in I;\ x_i\in W\}\in \mathcal{U}$. Since $W$ is arbitrary, we have $\displaystyle \lim_{i\to \mathcal{U}}x_i=x$.  
\end{proof}

\section{Unitary polar decomposition in $A'\cap A_{\omega}$}
The following result was communicated to the authors by Leonel Robert.  
\begin{proposition}\label{prop unitarypd} Let $A$ be a unital separable {\rm C}$^*$-algebra such that $A'\cap A_{\omega}$ has stable rank one. 
Then $A'\cap A_{\omega}$ admits unitary polar decomposition, i.e., for every $x\in A'\cap A_{\omega}$, there exists $u\in U(A'\cap A_{\omega})$ such that $x=u|x|$ holds. 
\end{proposition}
\begin{remark}
We were informed from Ilijas Farah that $A'\cap A_{\omega}$ is a 2-SAW$^*$ algebra in the sense of Pedersen \cite[$\S$3]{PedersenthreequaversMR0983335}, which by \cite[Theorem 3.5]{PedersenthreequaversMR0983335} implies that $A'\cap A_{\omega}$ admits unitary polar decomposition if $A'\cap A_{\omega}$ has stable rank one.  
This follows from the fact that $A'\cap A_{\omega}$ has quantifier-free saturation. See \cite{FarahbookMR3971570} for more results related to a variety of saturation conditions. 
\end{remark}
For the proof, we use the next lemma (see \cite[Lemma A.1]{kirchbergAbelMR2265050}), known as Kirchberg's $\varepsilon$-test:
\begin{lemma}[$\varepsilon$-test]\label{lem: epsilon test}
Let $\omega$ be a free ultrafilter on $\mathbb{N}$. Let $X_1,X_2,\dots$ be any sequence of sets. Suppose that for each $k\in \mathbb{N}$, we are given a sequence $(f_n^{(k)})_{n=1}^{\infty}$ of functions $f_n^{(k)}\colon X_n\to [0,\infty)$. For each $k\in \mathbb{N}$, define a new function $f_{\omega}^{(k)}\colon \prod_{n=1}^{\infty}X_n\to [0,\infty]$ by 
\[f_{\omega}^{(k)}(s_1,s_2,\dots)=\lim_{n\to \omega}f_n^{(k)}(s_n),\ \ \ (s_n)_{n=1}^{\infty}\in \prod_{n=1}^{\infty}X_n.\]
Suppose that for each $m\in \mathbb{N}$ and each $\varepsilon>0$, there exists a sequence $s=(s_1,s_2,\dots)\in \prod_{n=1}^{\infty}X_n$ such that 
\[f_{\omega}^{(k)}(s)<\varepsilon\text{\ \ \ \ for\ }k=1,2,\dots,m.\]
Then there exists a sequence $t=(t_1,t_2,\dots) \in \prod_{n=1}^{\infty}X_n$ with 
\[f_{\omega}^{(k)}(t)=0,\ \ \ \text{for all\ }k\in \mathbb{N}.\]
\end{lemma}
\begin{proof}[Proof of Proposition \ref{prop unitarypd}] 
Let $(x_n)_{n=1}^{\infty}$ be a bounded sequence in $A$ representing $x\in A'\cap A_{\omega}$.
Let $\{a_k\mid k\in \N\}$ be a countable dense subset of $A$. 
Define $X_n=U(A)$ for $n\in \N$. For each $k=0,1,2,\dots$, define $f_n^{(k)}\colon X_n\to [0,\infty]$ by 
\[f_n^{(0)}(u_n)=\|u_n|x_n|-x_n\|,\,f_n^{(k)}(u_n)=\|u_na_k-a_ku_n\|,\,\,k\in \N,\,u_n\in X_n.\]
Let $\varepsilon>0$ and $m\in \N$ be given. 
Since $A'\cap A_{\omega}$ has stable rank one, there exists a sequence $(y_j)_{j=1}^{\infty}$  of invertible elements in $A'\cap A_{\omega}$ such that $\|x-y_j\|\xrightarrow{j\to \infty}0$. Then $\||y_j|-|x|\|\xrightarrow{j\to \infty}0$ holds.  Let $v_j=y_j|y_j|^{-1}\in U(A'\cap A_{\omega})$. Then 
$y_j=v_j|y_j|$. 
We have
\eqa{
    \|v_j|x|-x\|&= \|y_j|y_j|^{-1}|x|-|x|\|\\
    &\le \|y_j|y_j|^{-1}(|x|-|y_j|)\|+\|y_j-x\|\\
    &=\||x|-|y_j|\|+\|y_j-x\|\xrightarrow{j\to \infty}0.
}
Therefore we may choose $j\in \N$ for which $\|v_j|x|-x\|<\varepsilon$ holds. Let $(u_n')_{n=1}^{\infty}$ be a sequence in $U(A)$ representing $v_j\in U(A'\cap A_{\omega})$. Then 
$f_{\omega}^{(k)}(u_1',u_2',\dots)<\varepsilon$ holds for $k=0,\dots,m$. 
By Lemma \ref{lem: epsilon test}, there exists $u_n\in X_n\,(n\in \N)$ such that $f^{(k)}_{\omega}(u_1,u_2,\dots)=0$ for every $k\in \N\cup \{0\}$. Thus, $u=(u_n)_{\omega}\in U(A'\cap A_{\omega})$ satisfies $x=u|x|$. 
\end{proof}

\section{Connectedness of $U(A'\cap A_{\omega})$ for AF algerbas}
We include the proof of the next result, which is likely a folklore. 
\begin{proposition}\label{prop F(AF) has connected unitary group}
Let $A$ be a separable unital AF algebra. Then 
$U(A'\cap A_{\omega})$ is connected.
\end{proposition}

Recall that by Russo--Dye theorem, for unital $C^*$-algebra $A$, any $a\in A$ with $\|a\|<1$ is a convex combination of unitaries. In particular, the closed convex hull of $U(A)$ coincides with the closed unit ball ${\rm Ball}(A)$. 
\begin{lemma}\label{lem stability EB}
Let $A$ be a unital $C^*$-algebra and $B\subset A$ a finite-dimensional unital $C^*$-subalgebra. 
Then there exists a conditional expectation $\E_B\colon A\to B'\cap A$ such that 
\begin{equation}\|a-\E_B(a)\|\le \max_{x\in {\rm Ball}(B)}\|ax-xa\|,\,\,a\in A.\label{eq distanceEB}
\end{equation}
\end{lemma}
\begin{proof}
Let $\mu_B$ be the normalized Haar measure on the compact group $U(B)$, and $x\in A$. 
Define $\E_B\colon A\to A$ by 
\[
    \E_B(x)=\int_{U(B)}uxu^*d\mu_B(u),\,\,x\in A.
\]
Note that $f_x\colon U(B)\ni u\mapsto uxu^*\in A$ is a continuous function taking values in the separable Banach space $A$. Thus, we may perform the Bochner integration of $f_x$ with respect to $\mu_B$, and thus $\E_B(x)\in A$ holds. It is clear that $\E_B$ is a bounded linear map and that $\|\E_B\|\le 1$. Let $v\in U(B)$. 
Then by the left-invariance of $\mu_B$, 
\eqa{
    v\E_B(x)v^*&=\int_{U(B)}vuxu^*v^*d\mu_B(u)=\int_{U(B)}uxu^*d\mu_B(v^*u)\\
    &=\int_{U(B)}uxu^*d\mu_B(u)
    =\E_B(x),
}
whence $\E_B(x)v=v\E_B(x)$ holds. Since $U(B)$ spans $B$, we obtain $\E_B(x)\in B'\cap A$. If moreover $x\in B'\cap A$, then because $uxu^*=x$ for every $u\in U(B)$, we obtain $\E_B(x)=\int_{U(B)}xd\mu_B(u)=x$. Thus, $\E_B$ is a norm one projection, which is therefore a conditional expectation of $A$ onto $B'\cap A$.\\
Next, we show (\ref{eq distanceEB}). Let $a\in A$. Then 
\eqa{
    \|a-\E_B(a)\|&=\left \|\int_{U(B)}(a-uau^*)d\mu_B(u)\right \|\\
    &\le \int_{U(B)}\|a-uau^*\|d\mu_B(u)=\int_{U(B)}\|au-ua\|d\mu_B(u)\\
    &\le \max_{u\in U(B)}\|au-ua\|.
}
Let $y\in {\rm Ball}(B)$ and $\varepsilon>0$. Then by Russo--Dye theorem, there exists $y_0$ of the form $y_0=\sum_{j=1}^m\lambda_ju_j$ for some $\lambda_1,\dots,\lambda_m>0,\,\lambda_1+\dots+\lambda_m=1$ and $u_1,\dots,u_m\in U(B)$, such that $2\|a\|\|y-y_0\|<\varepsilon$. 
Thus 
\eqa{
    \|ay_0-y_0a\|&\le \sum_{j=1}^m\lambda_j\|au_j-u_ja\|\\
    &\le \sum_{j=1}^m\lambda_j\max_{u\in U(B)}\|au-ua\|=\max_{u\in U(B)}\|au-ua\|.
}
Therefore 
\eqa{
    \|ay-ya\|&\le \|a(y-y_0)\|+\|ay_0-y_0a\|+\|(y_0-y)a\|\\
    &<\varepsilon+\max_{u\in U(B)}\|au-ua\|.
}
Since $\varepsilon>0$ is arbitrary and $U(B)\subset {\rm Ball}(B)$, we obtain $\disp \max_{x\in {\rm Ball}(B)}\|ax-xa\|=\max_{u\in U(B)}\|au-ua\|$. 
Therefore we obtain (\ref{eq distanceEB}).
\end{proof}
\begin{lemma}\label{lem cs lives in tail} Let $A$ be a separable unital AF algebra with $A_1\subset A_2\subset \cdots \subset A$ be an increasing sequence of 
finite-dimensional unital $C^*$-subalgebras such that $A=\overline{\bigcup_{n=1}^{\infty}A_n}$. 
Let $\E_n=\E_{A_n}\colon A\to A_n'\cap A$ be a conditional expectation given by Lemma \ref{lem stability EB} applied to the inclusion $A_n\subset A$. Then for every $a\in A'\cap A_{\omega}$, there exist $m_1, m_2, \cdots \in \N$ with $\lim_{n\to \omega}m_n=\infty$ and a bounded sequence $(a_n)_{n=1}^{\infty}\in \ell^{\infty}(A)$ such that $a_n\in A_{m_n}'\cap A\, (n\in \N)$ and $a=(a_n)_{\omega}$. 
\end{lemma}
\begin{proof}
We may assume that $\|a\|\le 1$, and let $(a_n^0)_{n=1}^{\infty}\in \ell^{\infty}(A)$ be a sequence of contractions such that $a=(a_n^0)_{\omega}$. Let $m\in \N$. Since $A_m$ is finite-dimensional, ${\rm Ball}(A_m)$ is compact. Therefore, there exists $d_m\in \N$ and $x_1^{(m)},\dots,x_{d_m}^{(m)}\in {\rm Ball}(A_m)$ such that 
$\min_{1\le k\le d_m}\|x-x_k^{(m)}\|<\frac{1}{m}$ for every $x\in {\rm Ball}(A_m)$. 
To find $m_n$'s we apply a technique similar to \cite[Lemma 3.13]{MR3198856AH}. 
Since $a\in A'\cap A_{\omega}$, we have 
\[J_m^0=\left \{n\in \N\,\middle|\,\left \|a_n^0x_k^{(m)}-x_k^{(m)}a_n^0\right \|<\frac{1}{m},\,k=1,\dots,d_m\right \}\in \omega.\]
Thus $J_m=\bigcap_{\ell=1}^mJ_{\ell}^0\in \omega$, and moreover $J_1\supset J_2\supset J_3\supset \cdots$ holds. Therefore, we have a following decomposition of $\N$ into disjoint subsets 
\[\N=\N\setminus J_1\sqcup \bigsqcup_{k=1}^{\infty}J_k\setminus J_{k+1}.\]
For $n\in J_1$, there exists unique $m_n\in \N$ such that $n\in J_{m_n}\setminus J_{m_n+1}$
Then $a_n=\E_{m_n}(a_n^0)\in A_{m_n}'\cap A$ satisfies $\|a_n\|\le 1$ for every $n\in \N$. For $n\in \N\setminus J_1$, we set $a_n=0$ and $m_n=1$.\\ \\
\textbf{Claim.} $\disp \lim_{n\to \omega}\|a_n^0-a_n\|=0$.\\
Let $\varepsilon>0$. Choose $m_0\in \N$ such that $\frac{1}{m_0}<\frac{\varepsilon}{3}$. 
Let $n\in J_{m_0+1}$. Then there exists unique $k\in \N$ such that $n\in J_{m_0+k}\setminus J_{m_0+k+1}$, whence $m_n=m_0+k$. 
Let $x\in {\rm Ball}(A_{m_n})$. Then there exists $j\in \{1,\dots,d_m\}$ such that $\|x-x_j^{(m_n)}\|<\frac{1}{m_n}$ holds. 
Also, by the definition of $J_{m_n}$, 
$\|a_n^0x_j^{(m_n)}-x_j^{(m_n)}a_n^0\|<\frac{1}{m_n}$ holds. 
It then follows that 
\eqa{
    \|[a_n^0,x]\|&\le \|a_n^0(x-x_j^{(m_n)})\|+\|[a_n^0,x_j^{(m_n)}]\|+\|(x_j^{(m_n)}-x)a_n^0\|\\
    &<\frac{3}{m_n}.
}
Since $x\in {\rm Ball}(A_{m_n})$ is arbitrary, thanks to Lemma \ref{lem stability EB}, we obtain 
\[\|a_n^0-\E_{m_n}(a_n^0)\|\le \max_{x\in {\rm Ball}(A_{m_n})}\|a_n^0x-xa_n^0\|\le \frac{3}{m_n}<\varepsilon.\]
Since $n\in J_{m_0+1}$ is arbitrary and $J_{m_0+1}\in \omega$, we obtain 
$\disp \lim_{n\to \omega}\|a_n^0-a_n\|\le \varepsilon$. Since $\varepsilon$ is arbitrary, we obtain $\disp \lim_{n\to \omega}\|a_n^0-a_n\|=0$. For each $m\in \N$, any $n\in J_m$ satisfies $m_n\ge m$, whence $\disp \lim_{n\to \omega}m_n=\infty$. This finishes the proof. 
\end{proof}

\begin{proof}[Proof of Proposition \ref{prop F(AF) has connected unitary group}]
Let $u\in U(A'\cap A_{\omega})$. Then by Lemma \ref{lem cs lives in tail}, 
there exist natural numbers $(m_n)_{n=1}^{\infty}$ with $\disp \lim_{n\to \omega}m_n=\infty$ and $(a_n)_{n=1}^{\infty}\in \ell^{\infty}(A)$ with $a_n\in A_{m_n}'\cap A,\,(n\in \N)$, such that $u=(a_n)_{\omega}$. Since $\lim_{n\to \omega}\|a_n^*a_n-1\|=0$, we may assume that each $|a_n|$ is invertible and thus $u_n=a_n|a_n|^{-1}\in U(A_{m_n}'\cap A)$ satisfies $u=(u_n)_{\omega}$. Since $A$ is AF and $A_{m_n}$ is finite-dimensional, $A_{m_n}'\cap A$ is also AF. Thus $U(A_{m_n}'\cap A)$ is connected and thus unitaries with finite spectrum are dense in its unitary group. Therefore we may assume that $u_n$ has finite spectrum. Thus $u_n=e^{\ri h_n}$ for some $h_n\in A_{m_n}'\cap A_{\rm sa}$ with $\|h_n\|\le \pi$. Note that $(h_n)_{\omega}\in A'\cap A_{\omega}$: let $x\in A$ be a contraction and $\varepsilon>0$. Then there exists $m_0\in \N$ and $x_0\in A_{m_0}$ such that $\|x-x_0\|<\frac{\varepsilon}{2\pi}$. 
Since $\lim_{n\to \omega}m_n=\infty$, the set $J=\{n\in \N\mid m_n\ge m_0\}$ belongs to $\omega$. For each $n\in J$, $A_{m_0}\subset A_{m_n}$, whence we have $h_n\in A_{m_n}'\cap A\subset A_{m_0}'\cap A$. Therefore 
\eqa{
    \|h_nx-xh_n\|&\le \|h_n(x-x_0)\|+\|h_nx_0-x_0h_n\|+\|(x_0-x)h_n\|\\
    &\le 2\|h_n\|\|x-x_0\|<\varepsilon.
}
This shows that $\disp \lim_{n\to \omega}\|h_nx-xh_n\|\le \varepsilon$, and since $\varepsilon$ is arbitrary, we obtain $h=(h_n)_{\omega}\in A'\cap A_{\omega}$. 
Thus, $u=e^{\ri h}$ is homotopic to 1 inside $A'\cap A_{\omega}$ by a continuous path $u(t)=e^{\ri th},\,t\in [0,1]$. 
\end{proof}


\section*{Acknowledgments}
We would like to thank Professors Hiroki Matui, Mikael R\o rdam for useful discussions on the group of approximately inner automorphisms. We are also grateful to the anonymous referee and professor Leonel Robert for their careful reading and spotting a mistake in the original manuscript.   
We also thank Professor Ilijas Farah for pointing out the undecidability result in Remark \ref{rem P onto undecidable} and for explaining the 2-SAW$^*$-property 
 for the central sequence algebra and Professor Robert for Proposition \ref{prop unitarypd} and the reference \cite{MR3908669BBSTWW}.\\  
H. Ando is supported by Japan Society for the Promotion of Sciences KAKENHI 20K03647. M. Doucha was supported by the GA\v CR project 22-07833K and by the Czech Academy of Sciences (RVO 67985840).
\bibliography{references} 
\end{document}